\documentclass{birkjour}
\usepackage{amssymb}
\usepackage{mathptmx}
\usepackage{graphicx}
\usepackage[mathcal]{euscript}

\usepackage[english]{babel}

\DeclareSymbolFont{SY}{U}{psy}{m}{n}
\DeclareMathSymbol{\emptyset}{\mathord}{SY}{'306}

\renewcommand{\eqref}[1]{{\rm(\ref{#1})}}

\newcommand{\bbR}{{\mathbb R}}
\newcommand{\bbN}{{\mathbb N}}

\newcommand{\cB}{{\mathcal B}}

\newcommand{\cG}{{\mathcal G}}
\newcommand{\cF}{{\mathcal F}}

\newcommand{\cO}{{\mathcal O}}

\newcommand{\cS}{{\mathcal S}}

\newcommand{\cU}{{\mathcal U}}

\newcommand{\re}{{\rm e}}

\newcommand{\sE}{{\sf E}}

\newcommand{\vk}{\varkappa}

\newcommand{\fA}{\mathfrak{A}}

\newcommand{\fH}{\mathfrak{H}}

\newcommand{\fL}{\mathfrak{L}}

\newcommand{\fP}{\mathfrak{P}}
\newcommand{\fQ}{\mathfrak{Q}}
\newcommand{\fR}{\mathfrak{R}}

\newcommand{\diag}{\mathop{\rm diag}}
\newcommand{\dist}{\mathop{\rm dist}}

\newcommand{\Real}{\mathop{\rm Re}}

\newcommand{\be}{\begin{equation}}
\newcommand{\ee}{\end{equation}}

 \DeclareMathOperator{\spec}{spec}

\newcommand{\ran}{\mathop{\mathrm{Ran}}}
\newcommand{\Ran}{\mathop{\mathrm{Ran}}}
\newcommand{\dom}{\mathop{\mathrm{Dom}}}
\newcommand{\Dom}{\mathop{\mathrm{Dom}}}

\allowdisplaybreaks

\numberwithin{equation}{section}

\newtheorem{theorem}{Theorem}[section]
\newtheorem{corollary}[theorem]{Corollary}
\newtheorem{lemma}[theorem]{Lemma}
\newtheorem{proposition}[theorem]{Proposition}

\theoremstyle{definition}
\newtheorem{definition}[theorem]{Definition}
\theoremstyle{remark}
{\it}{\rm}
\newtheorem{remark}[theorem]{Remark}

\begin{document}
\vspace*{-0cm}

\title[Sharpening the norm bound]
{Sharpening the Norm Bound in the Subspace \\ Perturbation
Theory$^*$\footnote{\normalsize$^*$arXiv:1112.0149;\\
\phantom{$^*$}\textit{Complex Analysis and Operator Theory}, DOI:
10.1007/s11785-012-0245-7;\\
\phantom{$^*$$^*$$^*$}the journal version is available online from\\
\phantom{$^*$$^*$$^*$}http://dx.doi.org/10.1007/s11785-012-0245-7}}

\author[S. Albeverio]{Sergio Albeverio}
\address{%
Institut f\"ur Angewandte Mathematik and HCM\\
Uni\-ver\-si\-t\"at Bonn\\
Endenicher Allee 60 \\
D-53115 Bonn, Germany\\
webpage: http://wiener.iam.uni-bonn.de/\~{}albeverio}
\email{\rm albeverio@uni-bonn.de}

\author[A. K. Motovilov]{Alexander K. Motovilov}
\address{Bogoliubov Laboratory of
Theoretical Physics\\
JINR, Joliot-Cu\-rie 6\\
141980 Dubna, Moscow
Region, Russia\\
webpage: http://theor.jinr.ru/\~{}motovilv}
\email{\rm motovilv@theor.jinr.ru}




\begin{abstract}
Let $A$ be a (possibly unbounded) self-adjoint operator on a
separable Hil\-bert space $\fH$. Assume that $\sigma$ is
an isolated component of the spectrum of $A$, that is,
$\dist(\sigma,\Sigma)=d>0$ where $\Sigma=\spec(A)\setminus\sigma$.
Suppose that $V$ is a bounded self-adjoint operator on $\fH$
such that $\|V\|<d/2$ and let $L=A+V$, $\Dom(L)=\Dom(A)$.
Denote by $P$ the spectral projection of $A$ associated
with the spectral set $\sigma$ and let $Q$ be the spectral
projection of $L$ corresponding to the closed $\|V\|$-neighborhood
of $\sigma$. Introducing the sequence
$$
\vk_n=\frac{1}{2}\left(1-\frac{(\pi^2-4)^n}{(\pi^2+4)^n}\right),
\quad n\in\{0\}\cup\bbN,
$$
we prove that the following bound holds:
$$
\arcsin(\|P-Q\|)\leq M_\star\left(\frac{\|V\|}{d}\right),
$$
where the estimating function $M_\star(x)$,
$x\in\bigl[0,\frac{1}{2}\bigr)$, is given by
$$
M_\star(x)=\frac{1}{2}\,\,n_{_\#}(x)\,\arcsin\left(\frac{4\pi}{\pi^2+4}\right)
+\frac{1}{2}\,\arcsin\left(\frac{\pi( x-\vk_{n_{_\#}(x)})}
{1-2\vk_{n_{_\#}(x)})}\right),
$$
with $n_{_\#}(x)=\max\left\{n\,\bigr|\,\,n\in\{0\}\cup\bbN\,,
\varkappa_n\leq x\right\}$. The bound obtained is essentially
stronger than the previously known estimates for $\|P-Q\|$.
Furthermore, this bound ensures that \mbox{$\|P-Q\|<1$} and, thus,
that the spectral subspaces $\Ran(P)$ and $\Ran(Q)$ are in the
acute-angle case whenever \mbox{$\|V\|<c_\star\,d$}, where
$$
c_\star=16\,\,\frac{\pi^6-2\pi^4+32\pi^2-32}{(\pi^2+4)^4}=0.454169\ldots\,.
$$
Our proof of the above results is based on using the triangle
inequality for the maximal angle between subspaces and on employing
the a priori generic $\sin2\theta$ estimate for the variation of a
spectral subspace. As an example, the boundedly perturbed quantum
harmonic oscillator is discussed.
\end{abstract}

\maketitle

\section{Introduction}
\label{SIntro}

One of fundamental problems of operator perturbation theory is to
study variation of the spectral subspace corresponding to a
subset of the spectrum of a closed linear operator that is
subject to a perturbation. This is an especially important issue
in perturbation theory of self-adjoint operators.

Assume that $A$ is a self-adjoint operator on a separable Hilbert
space $\fH$. It is well known (see, e.g., \cite{Kato}) that if $V$
is a bounded self-adjoint perturbation of $A$ then the spectrum,
$\spec(L)$, of the perturbed operator $L=A+V$, $\Dom(L)=\Dom(A)$,
lies in the closed $\|V\|$-neighborhood
$\cO_{\|V\|}\bigl(\spec(A)\bigr)$ of the spectrum of $A$. Hence,
if the spectrum of $A$ has an isolated component
$\sigma$ separated from its complement
$\Sigma=\spec(A)\setminus\sigma$  by gaps of length greater than or
equal to $d>0$, that is, if
\begin{equation}
\label{separIn}
\dist(\sigma,\Sigma)=d>0,
\end{equation}
then the spectrum of $L$ will also consist of two disjoint
components, $\omega=\spec(L)\cap\cO_{\|V\|}(\sigma)$ and
$\Omega=\spec(L)\cap\cO_{\|V\|}(\Sigma)$,
provided that
\begin{equation}
\label{Vd12In}
 \|V\|<{d}/{2}.
\end{equation}
Under condition \eqref{Vd12In} one may think of the separated
spectral components $\omega$ and $\Omega$ of the perturbed operator
$L$ as the result of the perturbation of the initial disjoint
spectral sets $\sigma$ and $\Sigma$, respectively. Clearly, this
condition is sharp in the sense that if $\|V\|>d/2$, the
perturbed operator $L$ may not have separated parts of the spectrum
at all.

Assuming \eqref{Vd12In}, let $P=\sE_A(\sigma)$ and $Q=\sE_L(\omega)$
be the spectral projections of the (self-adjoint) operators $A$ and
$L$ associated with the unperturbed and perturbed isolated spectral
sets $\sigma$ and $\omega$, respectively. A still unsolved problem
is to answer the following fundamental question:
\begin{enumerate}
\item[\it(i)] {\it Is it true
that under the single spectral condition \eqref{separIn} the
perturbation bound \eqref{Vd12In} necessarily implies
\begin{equation}
\label{PQless1}
\|P-Q\|<1\,?
\end{equation}
}
\end{enumerate}
Our guess is that the answer to the question (i) should be
positive, but at the moment this is only a conjecture.

Notice that the quantity $\theta=\arcsin\bigl(\|P-Q\|\bigr)$
expresses the maximal angle between the ranges $\Ran(P)$ and
$\Ran(Q)$ of the orthogonal projections $P$ and $Q$ (see Definition
\ref{D-maxangle} and Remark \ref{R-maxangle} below). If
$\theta<\pi/2$, the subspaces $\Ran(P)$ and $\Ran(Q)$ are said to be
in the acute-angle case. Thus, there is an equivalent geometric
formulation of the question (i): Does the perturbation bound
\eqref{Vd12In} together with the single spectral condition
\eqref{separIn} always imply that the spectral subspaces of $A$ and
$L$ associated with the respective unperturbed and perturbed
spectral sets $\sigma$ and $\omega$ are in the acute-angle case?

Furthermore, provided it is established that \eqref{PQless1} holds,
at least for
\begin{equation}
\label{Vdcd}
\|V\| < c\,d
\end{equation}
with some constant $c<1/2$, another important question arises:
\begin{enumerate}
\item[\it(ii)] {\it What function $M(x)$, $x\in[0,\,c)$,
is best possible
in the bound}
\begin{equation}
\label{Mquest}
\arcsin\bigl(\|P-Q\|\bigr)\leq M\left(\frac{\|V\|}{d}\right)\,?
\end{equation}
\end{enumerate}
The estimating function $M$ in \eqref{Mquest} is
required to be universal in the sense that it should be the same for
all self-adjoint $A$ and $V$ for which the conditions
\eqref{separIn} and \eqref{Vdcd} hold.

Note that if one adds to \eqref{separIn} an extra assumption that
one of the sets $\sigma$ and $\Sigma$ lies in a finite or infinite
gap of the other set, say, $\sigma$ lies in a gap of $\Sigma$, the
answer to the question (i) is known to be positive and the optimal
function $M$ in the bound \eqref{Mquest} is given by
$M(x)=\frac{1}{2}\arcsin(2x)$, $x\in\bigl[0,\frac{1}{2}\bigr)$. This
is the essence of the Davis-Kahan sin\,2$\Theta$ Theorem
\cite{DK70}. For the same particular mutual positions of the
spectral sets $\sigma$ and $\Sigma$ the positive answer to the
question (i) and complete answers to the question (ii) have also
been given (under  conditions on $\|V\|$ much weaker than
\eqref{Vd12In}) in the case where the perturbation $V$ is
off-diagonal with respect to the partition
\mbox{$\spec(A)=\sigma\cup\Sigma$} (see the $\tan2\Theta$ Theorem in
\cite{DK70} and the a priori $\tan\Theta$ Theorem in
\cite{AM2010,MotSel}; cf. the extensions of the $\tan2\Theta$
Theorem in \cite{GKMV2010,KMM5,MotSel}).

As for the general case where no requirements are imposed on the
spectral sets $\sigma$ and $\Sigma$ except for the separation
condition \eqref{separIn} and no assumptions are made on the
structure of the perturbation $V$, we are only aware of the partial
answers to the questions (i) and (ii) found in \cite{KMM1} and
\cite{MakS10}. We underline that both \cite{KMM1} and \cite{MakS10}
only treat the case where the unperturbed operator $A$ is bounded.
In \cite{KMM1} it was proven that inequality \eqref{PQless1} holds
true whenever \mbox{$\|V\|<c_{_{\rm KMM}}d$} with
\begin{align}
\label{cKMM}
c_{_{\rm KMM}}=&\frac{2}{2+\pi}=0.388984\ldots\,.
\end{align}
In \cite{MakS10} the value of $c$ in the bound \eqref{Vdcd} ensuring
\eqref{PQless1} has been raised to
\begin{align}
\label{cMS}
c_{_{\rm MS}}&=\frac{1}{2}-\frac{1}{2\re^2}=0.432332\ldots\,.
\end{align}
Explicit expressions for the corresponding estimating functions
$M$ found in \cite{MakS10} and \cite{KMM1} are given below in
Remarks \ref{R-MSbound} and \ref{R-KMMbound}, respectively. The
bound of the form \eqref{Mquest} established in  \cite{MakS10}
is stronger than its predecessor in \cite{KMM1}.

In the present work, the requirement that the operator $A$ should be bounded is
withdrawn. Introducing the sequence
$$
\vk_n=\frac{1}{2}\left(1-\frac{(\pi^2-4)^n}{(\pi^2+4)^n}\right),
\quad n\in\{0\}\cup\bbN,
$$
we prove that under conditions \eqref{separIn} and \eqref{Vd12In}
the following estimate of the form \eqref{Mquest} holds:
\begin{equation}
\label{PQMstar}
\arcsin(\|P-Q\|)\leq M_\star\left(\frac{\|V\|}{d}\right),
\end{equation}
where the estimating function $M_\star(x)$, $x\in\bigl[0,\frac{1}{2}\bigr)$, is given by
\begin{equation}
\label{MstarIn}
M_\star(x)=\frac{1}{2}\,\,n_{_\#}(x)\,\arcsin\left(\frac{4\pi}{\pi^2+4}\right)
+\frac{1}{2}\,\arcsin\left(\frac{\pi( x-\vk_{n_{_\#}(x)})}
{1-2\vk_{n_{_\#}(x)})}\right),
\end{equation}
with $n_{_\#}(x)=\max\left\{n\,\bigr|\,\,n\in\{0\}\cup\bbN\,,
\varkappa_n\leq x\right\}$. The estimate \eqref{PQMstar} is sharper
than the best previously known bound for $\|P-Q\|$ from
\cite{MakS10} (see Remark \ref{R-final} for details). Furthermore,
this estimate implies that \mbox{\,$\|P-Q\|<1$\,} and, thus,
that the spectral subspaces $\Ran(P)$ and $\Ran(Q)$ are in the
acute-angle case whenever \mbox{$\|V\|<c_\star d$}, where the constant
\begin{equation}
\label{cstarIn}
c_\star=16\,\,\frac{\pi^6-2\pi^4+32\pi^2-32}{(\pi^2+4)^4}=0.454169\ldots\,.
\end{equation}
is larger (and, hence, closer to the desired $1/2$) than the best
previous constant \eqref{cMS}.

The plan of the paper is as follows. In Section \ref{SecOR} we
recall the notion of maximal angle between subspaces of a Hilbert
space and recollect necessary definitions and facts on pairs of
subspaces. In Section \ref{S-MSext} we extend the best previously
known subspace perturbation bound (from \cite{MakS10}) to the case
where the unperturbed operator $A$ is already allowed to be
unbounded. Note that the extended bound is later on used in the
proof of the estimate \eqref{PQMstar}. Sec\-tion~\ref{S-Sin2T} is
devoted to deriving two new estimates for the variation of a
spectral subspace that we call the a priori and a posteriori generic
$\sin2\theta$ bounds. These respective estimates involve the maximal
angle between a reducing subspace of the perturbed operator $L$ and
a spectral subspace of the unperturbed operator $A$ or vice versa.
The principal result of the present work, the estimate \eqref{PQMstar},
is proven in Section \ref{S-StarBound} (see Theorem \ref{ThMain}).
Under the assumption that $V\neq0$, the proof is
performed by multiply employing the a priori generic $\sin2\theta$
bound and using, step by step, the triangle inequality for the
maximal angles between the corresponding spectral subspaces of the
operator $A$ and two consecutive intermediate operators $L_j=A+t_j
V$, where $t_j=\vk_j d/\|V\|$, $j\in\{0\}\cup\bbN$. Finally, in
Sec\-ti\-on~\ref{SecExHO} we apply the bound \eqref{PQMstar} to the
Schr\"odinger operator describing a boundedly perturbed
$N$-dimensional isotropic quantum harmonic oscillator.

The following notations are used thro\-ug\-h\-o\-ut the paper. By a
subspace of a Hilbert space we always mean a closed linear subset.
The identity operator on a subspace (or on the whole Hilbert space)
$\fP$ is denoted by $I_\fP$; if no confusion arises, the index $\fP$
is often omitted.  The Banach space of bounded linear operators from
a Hilbert space $\fP$ to a Hilbert space $\fQ$ is denoted by
$\cB(\fP,\fQ)$ and by $\cB(\fP)$ if $\fQ=\fP$. If $P$ is an
orthogonal projection in a Hilbert space $\fH$ onto the subspace
$\fP$, by $P^\perp$ we denote the orthogonal projection onto the
orthogonal complement $\fP^\perp:=\fH\ominus\fP$ of the subspace
$\fP$. The notation $\sE_T(\sigma)$ is used for the spectral projection
of a self-adjoint operator $T$ associated with a Borel set
$\sigma\subset\bbR$. By $\cO_r(\sigma)$, $r\geq 0$, we denote the
closed $r$-neigh\-bourhood of $\sigma$ in $\bbR$, i.e.\
$\cO_r(\sigma)=\{x\in\bbR\big|\,\dist(x,\sigma)\leq r\}$.

\section{Preliminaries}
\label{SecOR}

The main purpose of this section is to recollect relevant facts on
a pair of subspaces and the maximal angle between
them.

It is well known that if $\fH$ is a Hilbert space then $\|P-Q\|\leq
1$ for any two orthogonal projections $P$ and $Q$ in $\fH$ (see,
e.g., \cite[Section 34]{AkhiG}). We start with the following
definition.

\begin{definition}
\label{D-maxangle}
Let $\fP$ and $\fQ$ be subspaces of the Hilbert space $\fH$ and
$P$ and $Q$ the orthogonal projections in $\fH$
with $\Ran(P)=\fP$ and $\Ran(Q)=\fQ$. The quantity
$$
\theta(\fP,\fQ):=\arcsin(\|P-Q\|)
$$
is called the \textit{maximal angle} between the subspaces $\fP$ and $\fQ$.
\end{definition}
\begin{remark}
\label{R-maxangle}
The concept of maximal angle between subspaces is traced back to
M.G.\,Krein, M.\,A.\,Krasnoselsky, and D.\,P.\,Milman
\cite{Krein:Krasnoselsky:Milman}. Assuming that $(\fP,\fQ)$ is an
ordered pair of subspaces with $\fP\neq\{0\}$, they applied the
notion of the (relative) maximal angle between $\fP$ and $\fQ$ to the number
$\varphi$ in $\left[0,\text{$\frac{\pi}{2}$}\right]$
introduced by
$$
\sin\varphi(\fP,\fQ)=\sup\limits_{x\in\fP,\,\|x\|=1}\dist(x,\fQ).
$$
If both $\fP\neq\{0\}$ and $\fQ\neq\{0\}$ then
$$
\theta(\fP,\fQ)=\max\bigl\{\varphi(\fP,\fQ),
\varphi(\fQ,\fP)\bigr\}
$$
(see, e.g., \cite[Example 3.5]{BoetSpit}) and, in general,
$\varphi(\fP,\fQ)\neq \varphi(\fQ,\fP)$. Unlike
$\varphi(\fP,\fQ)$, the maximal angle $\theta(\fP,\fQ)$ is always symmetric
with respect to the interchange of the arguments $\fP$ and $\fQ$.
Furthermore,
$$
\varphi(\fQ,\fP)=\varphi(\fP,\fQ)=\theta(\fQ,\fP)\quad
\text{whenever  }\|P-Q\|<1.
$$
\end{remark}

\begin{remark}
The distance function $d(\fP,\fQ)=\|P-Q\|$ is a natural metric on
the set $\cS(\fH)$ of all subspaces of the Hilbert space
$\fH$. Clearly, for the maximal angle $\theta(\fP,\fQ)$ to be another metric on
$\cS(\fH)$, only the triangle inequality
\begin{equation*}
\theta(\fP,\fQ)\leq\theta(\fP,\fR)+\theta(\fR,\fQ),\quad\text{for any }
\fP,\fQ,\fR\in\cS(\fH),
\end{equation*}
is needed to be proven. That $\theta(\fP,\fQ)$ is indeed a metric on
$\cS(\fH)$ has been shown in \cite{Brown}. In Lemma \ref{L-triangle}
below we will give an alternative proof of this fact.
\end{remark}

\begin{remark}
\label{R-tPQperp}
$\theta(\fP^\perp,\fQ^\perp)=\theta(\fP,\fQ)$. This follows from
the equalities
$$
\|P^\perp-Q^\perp\|=\|(I-P)-(I-Q)\|=\|P-Q\|,
$$
where $I$ is the identity operator on $\fH$.
\end{remark}

\begin{definition}
Two subspaces $\fP$ and $\fQ$ of the Hilbert space $\fH$ are said to
be in the \emph{acute-angle case} if $\fP\neq\{0\}$, $\fQ\neq\{0\}$,
and $\theta(\fP,\fQ)<\frac{\pi}{2}$,  that is, if
\begin{equation}
\label{neq:PQ1}
\|P-Q\|<1,
\end{equation}
where $P$ and $Q$ are the orthogonal projections in $\fH$ with
$\Ran(P)=\fP$ and $\Ran(Q)=\fQ$.
\end{definition}

\begin{remark}
We recall that the subspaces $\fP$ and $\fQ$ are said to be in the
\emph{acute case} if $\fP\cap\fQ^\perp=\fP^\perp\cap\fQ=\{0\}$
(cf., e.g., \cite[Definition 3.1]{DK70}). The bound  $\|P-Q\|<1$ implies
both $\fP\cap\fQ^\perp=\{0\}$ and $\fP^\perp\cap\fQ=\{0\}$ (see, e.g.,
\cite[Theorem~2.2]{KMM2}). Hence, if the subspaces $\fP$ and $\fQ$
are in the acute-angle case, they are automatically in the
acute case.
\end{remark}

\begin{remark}
\label{R-Graph}
It is known (see, e.g., \cite[Corollary 3.4]{KMM2}) that
inequality \eqref{neq:PQ1} holds true (and, thus, $\fP$ and $\fQ$
are in the acute-angle case) if and only if the subspace $\fQ$ is the
graph of a bounded linear operator  $X$ from the subspace $\fP$ to its
orthogonal complement $\fP^\perp$, i.e.
\begin{equation}
\label{eq:G}
\fQ=\cG(X):=\{x\oplus Xx\,\,| x\in\fP\}\,.
\end{equation}
In such a case the projection $Q$ admits the representation
\begin{equation}
\label{eq:SP} Q=\begin{pmatrix} (I_{\fP} + X^*X)^{-1} &
(I_{\fP}+X^*X)^{-1}X^*\\ X(I_{\fP}+X^*X)^{-1} &
X(I_{\fP}+X^*X)^{-1}X^*
\end{pmatrix}
\end{equation}
with respect the orthogonal decomposition
$\fH=\fP\oplus\fP^\perp$ (cf. \cite[Remark 3.6]{KMM2}).
Moreover, under condition \eqref{neq:PQ1} the orthogonal projections
$P$ and $Q$ are unitarily equivalent. In particular,
\begin{equation*}
P=U^*QU,
\end{equation*}
where the unitary operator $U$ is given by
\begin{equation}
\label{eq:UT}
U=\begin{pmatrix}
(I_{\fP} + X^*X)^{-1/2} & -X^*(I_{\fP^\perp} + XX^*)^{-1/2}\\
X(I_{\fP} + X^*X)^{-1/2} & (I_{\fP^\perp} + XX^*)^{-1/2}
\end{pmatrix}.
\end{equation}
\end{remark}

\begin{remark}
One verifies by inspection that the unitary operator \eqref{eq:UT}
possesses the remarkable properties
\begin{equation}
\label{dir-rot}
U^2=(Q^\perp-Q)(P^\perp-P) \quad\text{and}\quad \Real U> 0,
\end{equation}
where $\Real U=\frac{1}{2}(U+U^*)$ denotes the real part of $U$.
\end{remark}

The concept of direct rotation from one subspace in the Hilbert
space to another was suggested by C. Davis in \cite{Davis}. The idea
of this concept goes back yet to  B. Sz.-Nagy (see \cite[\S
105]{SzNagy}) and T. Kato (see \cite[Sections I.4.6 and
I.6.8]{Kato}). We adopt the following definition of direct rotation
(see \cite[Proposition 3.3]{DK70}; cf. \cite[Definition
2.12]{MotSel}).

\begin{definition}
Let $\fP$ and $\fQ$ be subspaces of the Hilbert space $\fH$.
A unitary operator $S$ on $\fH$ is called the \emph{direct rotation}
from $\fH$ to $\fQ$ if
\begin{equation}
\label{drot-gen}
QS=SP,\quad S^2=(Q^\perp-Q)(P^\perp-P),\quad\text{and \,} \Real S\geq 0,
\end{equation}
where $P$ and $Q$ are the orthogonal projections in $\fH$ such that
$\Ran(P)=\fP$ and $\Ran(Q)=\fQ$.
\end{definition}

\begin{remark}
\label{RemU}
If the subspaces $\fP$ and $\fQ$ are not in the acute case,
the direct rotation from $\fP$ to $\fQ$ exists if and only if
$$
\dim(\fP\cap\fQ^\perp)=\dim(\fQ\cap\fP^\perp)
$$
(see \cite[Proposition 3.2]{DK70}). If it exists, it is not unique.
\end{remark}

\begin{remark}
\label{RemUex}
If the subspaces $\fP$ and $\fQ$ are in the acute case then there
exists a unique direct rotation from $\fP$ to $\fQ$ (see
\cite[Propositions 3.1 and 3.3]{DK70} or \cite[Theorem
2.14]{MotSel}). Comparing \eqref{dir-rot} with \eqref{drot-gen}, one
concludes that the unitary operator $U$ given by \eqref{eq:UT}
represents the unique direct rotation from the subspace $\fP$ to the
subspace $\fQ$ whenever these subspaces are in the acute-angle case.
The direct rotation $U$ has the
extremal property (see \cite[Theorem~7.1]{Davis})
\begin{equation}
\label{UIextr}
\|U-I_\fH\|=\inf_{\widetilde{U}\in\cU(\fP,\fQ)}\|\widetilde{U}-I_\fH\|,
\end{equation}
where $\cU(\fP,\fQ)$ denotes the set of all unitary operators
$\widetilde{U}$ on $\fH$ such that $P=\widetilde{U}^*Q\widetilde{U}$.
Equality \eqref{UIextr} says that the direct rotation $U$ is norm closest to the
identity operator among all unitary operators on $\fH$ mapping $\fP$ onto $\fQ$.
\end{remark}

The operator $X$ in the graph representation \eqref{eq:G} is usually
called the \emph{angular operator} for the (ordered) pair of the subspaces
$\fP$ and $\fQ$. The usage of this term is motivated by the equality
(see, e.g., \cite{KMM2})
\begin{equation}
\label{eq:XT}
\Theta(\fP,\fQ)=\arctan\sqrt{X^*X},
\end{equation}
where $\Theta(\fP,\fQ)$ denotes the \emph{operator angle} between the subspaces
$\fP$ and $\fQ$ (measured relative to the subspace~$\fP$). One
verifies by inspection (see, e.g., \cite[Corollary 3.4]{KMM2}) that,
in the acute-angle case,
\begin{equation}
\label{eq:NPD}
\sin\bigl(\theta(\fP,\fQ)\bigr)\equiv\|P-Q\|
=\frac{\|X\|}{\sqrt{1+\|X\|^2}}=\sin\bigl\|\Theta(\fP,\fQ)\bigr\|,
\end{equation}
which means, in particular, that
\begin{equation}
\label{tetTet}
\theta(\fP,\fQ)=\|\Theta(\fP,\fQ)\bigr\|.
\end{equation}
Furthermore, the lower bound for the spectrum of the real part
of the direct rotation \eqref{eq:UT} is given by
(cf. \cite[Remark 2.18]{MotSel})
\begin{equation}
\label{ReU}
\min\bigl(\spec(\Real U)\bigr)=\frac{1}{\sqrt{1+\|X\|^2}}=
\cos\bigl(\theta(\fP,\fQ)\bigr).
\end{equation}

To have a more convenient characterization of the distinction between a unitary
operator and the identity one, we recall the notion of the spectral angle.

\begin{definition}
\label{spang}
Let $S$ be a unitary operator.
The number
\begin{equation*}
\vartheta(S) = \sup_{z\in \spec(S)} |\arg z|, \quad \arg
z\in(-\pi,\pi],
\end{equation*}
is called the \textit{spectral angle} of $S$.
\end{definition}
\begin{remark}
The size of $\|S-I\|$ is easily computed in terms of $\vartheta(S)$ and vice versa
(see \cite[Lemma 2.19]{MotSel}). In particular,
\begin{equation}
\label{ImW} \| S - I \| = 2 \,\sin\left(\frac{\vartheta(S)}{2}\right).
\end{equation}
Furthermore,
\begin{equation}
\label{ReT}
\cos\vartheta(S)=\min\bigl(\spec(\Real S)\bigr).
\end{equation}
\end{remark}

By comparing \eqref{ReU} with \eqref{ReT} and \eqref{UIextr} with
\eqref{ImW} we immediately arrive at the following assertion.

\begin{proposition}
\label{PropU}
Let $\fP$ and $\fQ$ be subspaces of a Hilbert space $\fH$.
Assume that $\fP$ and $\fQ$ are in the acute-angle case. Then:
\begin{enumerate}
\item[\it (i)] The maximal angle between $\fP$ and $\fQ$ is nothing but
the spectral angle of the direct rotation $U$ from $\fP$ to $\fQ$, i.e.,
$
\theta(\fP,\fQ)=\vartheta(U).
$
\item[\it (ii)] $\vartheta(S)\geq\theta(\fP,\fQ)$
for any unitary $S$ on $\fH$ mapping $\fP$ onto $\fQ$.
\end{enumerate}
\end{proposition}

We conclude the present section with a proof of the triangle inequality
for the maximal angles between subspaces.

\begin{lemma}
\label{L-triangle}
Let $\fP$, $\fQ$, and $\fR$ be three arbitrary subspaces of the Hilbert space $\fH$.
The following inequality holds:
\begin{equation}
\label{eq:3ang}
\theta(\fP,\fQ)\leq\theta(\fP,\fR)+\theta(\fR,\fQ).
\end{equation}
\end{lemma}

\begin{proof}
If $\theta(\fP,\fR)+\theta(\fR,\fQ)\geq\pi/2$, inequality
\eqref{eq:3ang} holds true since $\theta(\fP,\fQ)\leq\pi/2$ by the
definition of the maximal angle.

Suppose that $\theta(\fP,\fR)+\theta(\fR,\fQ)<\pi/2$. In such a case
both the pairs $(\fP,\fR)$ and $(\fR,\fQ)$  of the argument
subspaces are in the acute-angle case. Then there are a unique direct
rotation $U_1$ from $\fP$ to $\fR$ and a unique direct rotation
$U_2$ from $\fR$ to $\fQ$ (see Remark \ref{RemUex}). By \cite[Lemma
2.22]{MotSel}, the spectral angle $\vartheta(S)$ of the product
$S:=U_2U_1$ of the unitary operators $U_1$ and $U_2$ satisfies the
bound
\begin{align}
\label{tStUtU}
\vartheta(S)\leq&\vartheta(U_1)+\vartheta(U_2).
\end{align}
Notice that  by Proposition \ref{PropU}\,(i)
\begin{equation}
\label{tUtU}
\vartheta(U_1)=\theta(\fP,\fR)\text{\, and \,}
\vartheta(U_2)=\theta(\fR,\fQ),
\end{equation}
because both $U_1$ and $U_2$ are direct rotations. Since
$\Ran\bigl(U_1|_{\fP}\bigr)=\fQ$ and
$\Ran\bigl(U_2|_{\fQ}\bigr)=\fR$, the unitary operator $S$ maps
$\fP$ onto $\fQ$. Hence, $\vartheta(S)\geq\theta(\fP,\fQ)$ by
Proposition \ref{PropU}\,(ii). Combining this with \eqref{tStUtU}
and \eqref{tUtU} completes the proof.
\end{proof}

\section{An extension of the best previously known bound}
\label{S-MSext}

In this section we extend the norm estimate on variation of spectral
subspaces of a bounded self-adjoint operator $A$ under a bounded
self-adjoint perturbation $V$ established recently by K.\,A.\,Ma\-ka\-rov and
A.\,Se\-el\-mann in \cite{MakS10} to the case where $A$ is allowed to
be unbounded.

We begin with recalling the concept of a strong solution to the
operator Sylvester equation.

\begin{definition}
\label{DefSolSyl}
Let $\Lambda_0$ and $\Lambda_1$ be (possibly
unbounded) self-adjoint operators on the Hilbert spaces $\fH_0$ and
$\fH_1$, respectively, and $Y \in \cB(\fH_0, \fH_1)$.
A bounded operator $X\in\cB(\fH_0,\fH_1)$ is said to be a \emph{strong
solution} to the Sylvester equation
\begin{equation}
\label{SylEq} X\Lambda_0-\Lambda_1X=Y
\end{equation}
if
\begin{equation*}
\ran\bigl({X}|_{\dom(\Lambda_0)}\bigr)\subset\dom(\Lambda_1)
\end{equation*}
and
\begin{equation*}
X\Lambda_0f-\Lambda_1Xf=Yf \quad  \text{ for all \,} f\in
\dom(\Lambda_0).
\end{equation*}
\end{definition}

We will use the following well known result on a sharp norm bound
for strong solutions to operator Sylvester equations (cf.\,
\cite[The\-o\-rem~4.9 (i)]{AlMoSh}).

\begin{theorem}
\label{TSylB} Let $\Lambda_0$, $\Lambda_1$, and $Y$ be as
in Definition \ref{DefSolSyl}. If the spectra of
$\Lambda_0$ and $\Lambda_1$ are disjoint, i.e. if
$$
\delta:=\dist\bigl(\spec(\Lambda_0),\spec(\Lambda_1)\bigr)>0,
$$
then the Sylvester equation \eqref{SylEq}
has a unique strong solution $X\in\cB(\fH_0,\fH_1)$. Moreover, the
solution $X$ satisfies the bound
\begin{equation}
\label{XBg}
\|X\|\leq\frac{\pi}{2}\,\,\frac{\|Y\|}{\delta}.
\end{equation}
\end{theorem}

\begin{remark}
The fact that the constant $c$ in the estimate $\delta\|X\|\leq c
\|Y\|$ for the
generic disposition of the spectra of $\Lambda_0$ and $\Lambda_1$ is
not greater than $\pi/2$ goes back to  B.~Sz.-Nagy and A.~Strausz
(see \cite{SN53}). The sharpness of the value $c=\pi/2$ has been
proven by R.\,McEa\-chin \cite{McE93}. In its present form the
statement is obtained by combining \cite[Theorem 2.7]{AMM} and
\cite[Lemma 4.2]{AM01}.
\end{remark}

The next statement represents nothing but a corollary to Theorem \ref{TSylB}.

\begin{proposition}{\rm(cf. \cite{McE93})}
\label{PropS}
Let  $A$ and $B$ be possibly unbounded self-adjoint operators on the
Hilbert space $\fH$ with the same domain, i.e. $\Dom(B)=\Dom(A)$.
Assume that the closure $C=\overline{B-A}$ of the symmetric operator
$B-A$ is a bounded self-adjoint operator on $\fH$. Then for any two
Borel sets $\omega,\Omega\subset\bbR$ the following
inequality holds:
\begin{equation}
\label{SylEE}
\dist(\omega,\Omega)\,\|\sE_A(\omega)\sE_B(\Omega)\|\leq\frac{\pi}{2}\|C\|,
\end{equation}
where $\sE_A(\omega)$ and $\sE_L(\Omega)$ are the spectral projections
of $A$ and $B$ associated with the sets
$\omega$ and $\Omega$, respectively.
\end{proposition}
\begin{proof}
Clearly, it suffices to give a proof only for the case where
\begin{equation}
\label{AssAL}
\omega\subset\spec(A), \quad\Omega\subset\spec(B),\text{\, and \,}
\dist(\omega,\Omega)>0.
\end{equation}
Assuming \eqref{AssAL}, we set $P=\sE_A(\omega)$ and
$Q=\sE_B(\Omega)$. The spectral theorem implies
\begin{align}
\label{DAP}
Pf&\in\Dom(A)\cap\fP\text{\,\, for any }f\in\Dom(A),\\
\label{DLQ}
Qg&\in\Dom(B)\cap\fQ\text{\,\, for any }g\in\Dom(B),
\end{align}
where $\fP:=\Ran(P)$ and $\fQ:=\Ran(Q)$ are the spectral
subspaces of the operators $A$ and $B$ associated with their
respective spectral subsets $\omega$ and $\Omega$. Due to
$\Dom(B)=\Dom(A)$ the inclusions \eqref{DAP} and \eqref{DLQ} yield
\begin{equation}
\label{PQin}
\Ran\left(PQ|_{\Dom(B)}\right)\subset\Dom(A)\cap\fP.
\end{equation}
Since $P$ commutes with $A$, $Q$ commutes with $B$, $P^2=P$, and $Q^2=Q$,
from \eqref{DAP}--\eqref{PQin} it follows that
\begin{equation}
\label{SylAux}
PQ\,\,QBQf-PAP\,\,PQf=P\,C\,Qf\quad\text{ for any }f\in\Dom(B).
\end{equation}

Now let $A_\omega$ and $B_\Omega$ be the parts of the self-adjoint
operators $A$ and $B$ associated with their spectral subspaces
$\fP=\Ran(P)$ and $\fQ=\Ran(Q)$. That is,
\begin{align}
\label{DAom}
A_\omega&=A|_{\fP}\,\,\text{ with \,}\Dom(A_\omega)=\fP\cap\Dom(A),\\
\label{DLOm}
B_\Omega&=B|_{\fQ}\,\,\text{ with \,}\Dom(B_\Omega)=\fQ\cap\Dom(B)
\quad\bigl(=\fQ\cap\Dom(A)\bigr).
\end{align}
Also set $X:=P|_{\fQ}=PQ|_{\fQ}$. Taking into account \eqref{PQin},
\eqref{DAom}, and \eqref{DLOm} we have
$$
\Ran\left(X|_{\Dom(B_\Omega)}\right)\subset\Dom(A_\omega)
$$
and then \eqref{SylAux} implies
$$
XB_\Omega f-A_\omega X f=PCf\quad \text{ for any }f\in\Dom(B_\Omega),
$$
which means that the operator $X$ is a strong solution to the operator Sylvester equation
$$
XB_\Omega -A_\omega X =PC|_{\fQ}.
$$
To prove \eqref{SylEE} it only remains to notice that
\begin{align*}
\|\sE_A(\omega)\sE_B(&\Omega)\|=\|PQ\|=\|X\|,\\
\spec(A_\omega)=\omega,\,\,\,\, &\spec(B_\Omega)=\Omega,\,\,\,\, \|PC|_{\fQ}\|\leq\|C\|
\end{align*}
and then to apply Theorem \ref{TSylB}.
\end{proof}

\begin{theorem}
\label{ThGlog} Given a (possibly unbounded) self-adjoint operator
$A$ on the Hilbert space $\fH$, assume that a Borel set
$\sigma\subset\bbR$ is an isolated component of the spectrum of $A$,
that is, \mbox{$\sigma\subset\spec(A)$} and
\begin{equation}
\label{dsig}
\dist(\sigma,\Sigma)=d>0,
\end{equation}
where $\Sigma=\spec(A)\setminus\sigma$. Assume, in addition, that $V$ is a
bounded self-adjoint operator on $\fH$ such that
\begin{equation}
\label{dV2}
\|V\|<{d}/{2}
\end{equation}
and let $\Gamma(t)$,
$t\in[0,1]$, be the spectral projection of the self-adjoint
operator
\begin{equation}
\label{Lt}
L_t=A+tV, \quad \Dom(L_t)=\Dom(A),
\end{equation}
associated with the closed
$\|V\|$-neighborhood $\cO_{\|V\|}(\sigma)$ of the set
$\sigma$.
The projection family $\{\Gamma(t)\}_{t\in[0,1]}$
is norm continuous on the interval $[0,1]$ and
\begin{equation}
\label{eq:MSb} \arcsin\left(\|\Gamma(b)-\Gamma(a)\|\right) \leq\frac{\pi}{4}
\log\left(\frac{d-2a\|V\|}{d-2b\|V\|}\right) \quad\text{whenever \,\,}
0\leq a<b\leq 1.
\end{equation}
\end{theorem}

\begin{proof}

Let $\omega_t=\spec(L_t)\cap\cO_{\|V\|}(\sigma)$ and
$\Omega_t=\spec(L_t)\cap\cO_{\|V\|}(\Sigma)$, $t\in[0,1]$. Since $A$
is a self-adjoint operator and $L_t$ is given by \eqref{Lt}, we have
\begin{equation*}
\omega_t\cap\Omega_t=\emptyset, \quad \spec(L_t)=\omega_t\cup\Omega_t,
\end{equation*}
and, in fact,
\begin{equation}
\label{omOmt}
\omega_t\subset\cO_{t\|V\|}(\sigma)\text{\,\, and
\,\,}\Omega_t\subset\cO_{t\|V\|}(\Sigma),\quad t\in[0,1].
\end{equation}
Notice that under condition \eqref{dV2} from \eqref{dsig} and
\eqref{omOmt} it follows that, for $s,t\in[0,1]$,
\begin{align}
\label{domOmt}
\dist(\omega_t,\Omega_s)\geq d-t\|V\|-s\|V\|=&d-\|V\|(t+s).
\end{align}
In particular,
\begin{align}
\label{dVts}
d-\|V\|(t+s)&\geq d-2\|V\|>0\quad\text{whenever \,\,}s,t\in[0,1].
\end{align}

Obviously, $\Gamma(t)=\sE_{L_t}(\omega_t)$ and
$\Gamma(t)^\perp=\sE_{L_t}(\Omega_t)$, $t\in[0,1]$. Thus, for any $s,t\in[0,1]$
Proposition \ref{PropS} implies
\begin{align}
\label{BPBP}
\|\Gamma(s)\Gamma(t)^\perp\|&\leq
\frac{\pi}{2}\,\frac{\|V\||t-s|}{\dist(\omega_s,\Omega_t)} \text{\,\, and \,\,}
\|\Gamma(s)^\perp\Gamma(t)\|\leq
\frac{\pi}{2}\,\frac{\|V\||t-s|}{\dist(\Omega_s,\omega_t)}.
\end{align}
Since $\|\Gamma(t)-\Gamma(s)\|=
\max\bigl\{\|\Gamma(t)\Gamma(s)^\perp\|,\|\Gamma(t)^\perp\Gamma(s)\|\bigr\}$,
from \eqref{domOmt} and \eqref{BPBP} one concludes that
\begin{equation}
\label{GsGt}
\|\Gamma(t)-\Gamma(s)\|\leq \frac{\pi}{2}\,\frac{\|V\||t-s|}{d-\|V\|(t+s)}
\quad \text{for any \,\,}s,t\in[0,1].
\end{equation}
In view of \eqref{dVts}, the operator norm continuity of the
projection path $\{\Gamma(t)\}_{t\in[0,1]}$ on the interval $[0,1]$
follows immediately from estimate \eqref{GsGt}.

Now suppose that $s<t$ (as before,  $s,t\in[0,1]$)
and observe that
\begin{equation}
\label{neq:c}
\frac{t-s}{d-\|V\|(t+s)}<\int_s^t \frac{d\tau}{d-2\|V\|\tau}.
\end{equation}
Indeed, the difference between the right-hand side and left-hand side parts of
\eqref{neq:c} may be written as
\begin{equation}
\label{Intl}
\int_s^{\tau_c} \bigl(f(\tau)+f(2\tau_c-\tau)-2f(\tau_c)\bigr)d\tau,
\end{equation}
where $\tau_c=(s+t)/2$ is the center of the interval $[s,t]$ and
$$
f(\tau):=\frac{1}{d-2\|V\|\tau},\,\, \tau\in[s,t].
$$
One verifies by inspection that the expression
$f(\tau)+f(2\tau_c-\tau)-2f(\tau_c)$ under the integration sign in
\eqref{Intl} is strictly positive for $\tau\in[s,\tau_c)$ and zero
for $\tau=\tau_c$. Therefore, the integral in \eqref{Intl} is
positive and hence inequality \eqref{neq:c} holds true.

Assume that $0\leq a<b\leq 1$. For a sequence of points
$t_0,t_1,\ldots,t_n\in[a,b]$, $n\in\bbN$, such that
\begin{equation}
\label{atb}
a=t_0<t_1<\ldots<t_n=b
\end{equation}
by \eqref{GsGt} and \eqref{neq:c} one obtains
\begin{align}
\nonumber
\sum\limits_{j=0}^{n-1}\,\|\Gamma(t_{j+1}-\Gamma(t_j)\|\leq&
\frac{\pi\|V\|}{2}\,\,\sum\limits_{j=0}^{n-1}\,\,\frac{t_{j+1}-t_j}{d-\|V\|(t_j+t_{j+1})}\\
\nonumber
<&
\frac{\pi\|V\|}{2}\,\,\sum\limits_{j=0}^{n-1}\,\,
\int_{t_j}^{t_{j+1}} \frac{d\tau}{d-2\|V\|\tau}\\
\label{Intf}
&=\frac{\pi\|V\|}{2}\int_a^b \frac{d\tau}{d-2\|V\|\tau}.
\end{align}
Evaluating the last integral in \eqref{Intf} and taking supremum over all choices
of $n\in\bbN$ and $t_0,t_1,\ldots,$ $t_n\in[a,b]$ satisfying \eqref{atb}
results in the bound
\begin{equation*}
\ell(\Gamma)\leq\frac{\pi}{4}
\log\left(\frac{d-2a\|V\|}{d-2b\|V\|}\right),
\end{equation*}
where
\begin{equation}
\label{lpath}
\ell(\Gamma):=\sup\left\{
\sum\limits_{j=0}^{n-1}\,\|\Gamma(t_{j+1})-\Gamma(t_j)\|\,\,\biggl|\,\,
n\in\mathbb{N},\,\, a=t_0<t_1<\ldots<t_n=b
\right\}
\end{equation}
is the length of the (continuous) projection path $\Gamma(t)$, $t\in[a,b]$.
Applying \cite[Corollary 4.2]{MakS10}, which establishes that
$
\arcsin(\|\Gamma(b)-\Gamma(a)\|)\leq \ell(\Gamma),
$
completes the proof.
\end{proof}

\begin{remark}
\label{R-MSbound}
Let $\fA=\Ran\bigl(\sE_A(\sigma)\bigr)$ and
$\fL=\Ran\bigl(\sE_L(\omega)\bigr)$ where $L:=A+V$ with
$\Dom(L)=\Dom(A)$ and $\omega=\spec(L)\cap\cO_{\|V\|}(\sigma)$. By
setting $a=0$ and $b=1$ in \eqref{eq:MSb}, one obtains
\begin{equation}
\label{MSbound}
\theta(\fA,\fL)\leq M_{_{\rm MS}}\left(\frac{\|V\|}{d}\right),
\end{equation}
where $\theta(\fA,\fL)$ is the maximal angle between the spectral
subspaces $\fA$ and $\fL$ and
\begin{equation}
\label{MMS}
M_{_{\rm MS}}(x):
=\frac{\pi}{4}\,\log\left(\frac{1}{1-2x}\right),\quad
x\in\bigl[0,\mbox{\small$\frac{1}{2}$}\bigr).
\end{equation}
For bounded $A$, the estimate \eqref{MSbound} has been established
in \cite[Theorem 6.1]{MakS10}.

Note that \eqref{MSbound} implies that
$\theta(\fA,\fL)<\frac{\pi}{2}$ and, thus, that the subspaces $\fA$
and $\fL$ are in the acute-angle case if
\begin{equation}
\label{VMS}
\|V\|<c_{_{\rm MS}} d,
\end{equation}
where $c_{_{\rm MS}}$ is given by \eqref{cMS}.
\end{remark}

\begin{remark}
\label{R-MSpi4}
For future references we remark that, due to \eqref{MSbound},
\begin{equation}
\label{teta-pi4}
\theta(\fA,\fL)\leq\frac{\pi}{4}\qquad\text{whenever\,\,\,
}\|V\|\leq c_{_{\pi/4}}\,d,
\end{equation}
where
\begin{equation}
\label{cpi4}
c_{_{\pi/4}}=\frac{1}{2}-\frac{1}{2\re}=0.316060\ldots\,.
\end{equation}
\end{remark}

\begin{remark}
\label{R-KMMbound}
The bound \eqref{MSbound} is stronger than the earlier estimate \cite{KMM1}
\begin{equation}
\label{KMMbound} \theta(\fA,\fL)\leq M_{_{\rm
KMM}}\left(\frac{\|V\|}{d}\right)\quad\left(<\frac{\pi}{2}\right)
\qquad \text{if \,\,}\|V\|<c_{_{\rm KMM}} d,
\end{equation}
where the value of $c_{_{\rm KMM}}$ is given by \eqref{cKMM} and
\begin{align}
\label{MKMM}
M_{_{\rm KMM}}(x):=&\arcsin\left(\frac{\pi
x}{2(1-x)}\right), \quad 0\leq x\leq c_{_{\rm KMM}}.
\end{align}
The estimate \eqref{KMMbound} was established in the proof of Lemma
2.2 in \cite{KMM1}.
\end{remark}

\section{A priori and a posteriori generic sin\,$2\theta$ estimates}
\label{S-Sin2T}

We begin this section with the proof of an estimate for
$\sin\bigl(2\theta(\fA,\fL)\bigr)$, where $\fA$ is a reducing
subspace of the self-adjoint operator $A$ and $\fL$ is the spectral
subspace of the boundedly perturbed self-adjoint operator $L=A+V$.
In general, $\fA$ does not need to be a spectral subspace of $A$.

\begin{theorem}
\label{ThS2Tpo} Let $A$ be a (possibly unbounded) self-adjoint
operator on the Hilbert space $\fH$. Suppose that $V$ is a bounded
self-adjoint operator on $\fH$ and $L=A+V$ with $\Dom(L)=\Dom(A)$.
Assume that $\mathfrak{A}$ is a reducing subspace of $A$ and $\fL$
is a spectral subspace of $L$ associated with a Borel
subset $\omega$ of its spectrum. If the subspaces $\fA$ and $\fL$ are in the
acute-angle case then
\begin{equation}
\label{s2t-apo}
\dist(\omega,\Omega)\,\sin(2\theta)\leq \pi \|V\|,
\end{equation}
where $\Omega=\spec(L)\setminus\omega$ is the remainder of the
spectrum of $L$ and $\theta:=\theta(\fA,\fL)$ denotes the maximal angle between
$\fA$ and $\fL$.
\end{theorem}

\begin{proof}
For \,$\dist(\omega,\Omega)=0$\, the bound \eqref{s2t-apo} is trivial. Throughout
the proof below we will assume that \,$\dist(\omega,\Omega)\neq 0$.

Since $\fA$ is a reducing subspace of the self-adjoint operator $A$,
its orthogonal complement $\fA^\perp=\fH\ominus\fA$ is also a
reducing subspace of $A$. Furthermore,
$\Dom(A)=\Dom(A_0)\oplus\Dom(A_1)$, where $\Dom(A_0)$ and $\Dom(A_1)$
are domains of the parts $A_0=A|_\fA$ and $A_1=A|_{\fA^\perp}$ of
$A$ in its reducing subspaces $\fA$ and $\fA^\perp$, respectively
(see, e.g., \cite[Section 3.6]{BirSol}).

Let $P$ be the orthogonal projection in $\fH$ onto the subspace
$\fA$. Since $V\in\cB(\fH)$, the (self-adjoint) operator
$L$ admits the following block representation with
respect to the orthogonal decomposition $\fH=\fA\oplus\fA^\perp$:
\begin{align}
\label{L}
L & = \left(\begin{array}{cc} D_0 & B\\
B^* & D_1
\end{array}\right), \quad \Dom(L)=\Dom(D_0)\oplus\Dom(D_1)
\quad\bigl(=\Dom(A)\bigr),
\end{align}
where $B=PV|_{\fA^\perp}$,\, $D_0=A_0+PV|_\fA$, with
$\Dom(D_0)=\Dom(A_0)$ and $D_1=A_1+P^\perp V|_{\fA^\perp}$, with
$\Dom(D_1)=\Dom(A_1)$.

That the subspaces $\fA$ and $\fL$ are in the acute-angle case
implies that there is a bounded operator $X$ from $\fA$ to
$\fA^\perp$ such that $\fL$ is the graph of $X$, that is,
$\fL=\cG(X)$ (see Remark \ref{R-Graph}). It is well known (see,
e.g., \cite[Lemma 5.3]{AMM}) that the graph subspace $\cG(X)$ is a
reducing subspace for the block operator matrix \eqref{L} if and
only if the angular operator $X$ is a strong solution to the
operator Riccati equation
\begin{equation}
\label{Ric0} XD_0-D_1 X+XBX=B^*.
\end{equation}
The notion of the strong solution to \eqref{Ric0} means that (see
\cite{AMM,AM01}; cf. Definition \ref{DefSolSyl})
\begin{equation}
\label{ranric} \ran\bigl({X}|_{\Dom(D_0)}\bigr)\subset\dom(D_1)
\end{equation}
and
\begin{equation}
\label{rics}
XD_0f-D_1Xf+XBXf=B^*f  \text{ \, for all \, } f\in \dom(D_0).
\end{equation}
It is straightforward to verify that if $X$ is a strong solution to
\eqref{rics} then
\begin{equation}
\label{RicX} XZ_0f-Z_1Xf=B^*(I+X^*X)f\quad \text{for
all \,}f\in\Dom(D_0),
\end{equation}
where $Z_0=D_0+BX$ with $\dom(Z_0)=\dom(D_0)=\Dom(A_0)$ and $Z_1=D_1-B^*X^*$
with $\dom(Z_1)=\dom(D_1)=\dom(A_1)$.

Our next step is in transforming \eqref{RicX} into
\begin{align}
\nonumber & X(I+X^*X)^{-1/2}\Lambda_0(I+X^*X)^{1/2}f-
(I+XX^*)^{-1/2}\Lambda_1(I+XX^*)^{1/2}Xf\qquad\\
\label{RicXX} &\qquad\qquad\qquad \qquad\qquad =B^*(I+X^*X)f \quad
\text{for all }f\in\Dom(D_0),
\end{align}
where $\Lambda_0$ and $\Lambda_1$ are given by
\begin{equation}\label{Lam0}
\begin{array}{l}
\Lambda_0 =(I+X^*X)^{1/2}Z_0(I+X^*X)^{-1/2},\\
\quad\quad
\Dom(\Lambda_0)=\Ran\Bigl((I+X^*X)^{1/2}\bigr|_{\Dom(D_0)}\Bigr),
\end{array}
\end{equation}
and
\begin{equation}
\label{Lam1}
\begin{array}{l}
\Lambda_1 = (I+XX^*)^{1/2}Z_1(I+XX^*)^{-1/2},\\
\quad\quad
\Dom(\Lambda_1)=\Ran\Bigl((I+XX^*)^{1/2}\bigr|_{\Dom(D_1)}\Bigr).
\end{array}
\end{equation}
Note that the (self-adjoint) operator $L$ is unitary equivalent to
the block diagonal operator matrix
$\Lambda=\diag(\Lambda_0,\Lambda_1)$,
$\Dom(\Lambda)=\Dom(\Lambda_0)\oplus\Dom(\Lambda_1)$ (see, e.g.,
\cite[Theorem 5.5]{AMM}) and, thus, both $\Lambda_0$ and $\Lambda_1$
are self-adjoint operators. Furthermore, from \eqref{Lam0} and
\eqref{Lam1} it follows that $Z_0$ and $Z_1$ are similar to
$\Lambda_0$ and $\Lambda_1$, respectively. Combining \cite[Theorem
5.5]{AMM} with \cite[Corollary 2.9 (ii)]{AlMoSh} then yields
\begin{equation}
\label{spLL}
\spec(\Lambda_0)=\spec(L|_{\fL})=\omega\quad\text{and}\quad
\spec(\Lambda_1)=\spec(L|_{\fL^\perp})=\Omega.
\end{equation}

Applying $(I+XX^*)^{1/2}$ from the left to both sides of \eqref{RicX}
and choosing $f=(I+X^*X)^{-1/2}g$ with $g\in\Dom(\Lambda_0)$, we
arrive at the Sylvester equation
\begin{equation}
\label{XKSyl} K\Lambda_0g-\Lambda_1 Kg=Yg \quad \text{for all }
g\in\dom(\Lambda_0),
\end{equation}
where
\begin{align}
\label{XX}
K=&\ \, (I+XX^*)^{1/2}X(I+X^*X)^{-1/2},\qquad\\
\label{XY} Y=&(I+XX^*)^{1/2}B^*(I+X^*X)^{1/2}.
\end{align}
By \eqref{Lam0} we have
$\Ran\biggl((I+X^*X)^{-1/2}\bigl|_{\Dom(\Lambda_0)}\biggr)=\Dom(D_0)$.
Furthermore, $$\ran\bigl({X}|_{\dom(D_0)}\bigr)\subset\dom(D_1)$$ by
\eqref{ranric}, and thus, by \eqref{Lam1},
\begin{equation}
\label{DomK}
\Ran\bigl(K\bigr|_{\dom(\Lambda_0)}\bigr)\subset\dom(\Lambda_1).
\end{equation}
Hence $K$ is a strong solution to the Sylvester equation
\eqref{XKSyl}.

It is easy to verify (see \cite[Lemma 2.5]{AMT09}) that
$(I+XX^*)^{1/2}X=X(I+X^*X)^{1/2}$. Thus, \eqref{XX} simplifies
to nothing but the identity $X=K$, which by \eqref{XKSyl} and
\eqref{DomK} means that $X$ is a strong solution to the Sylvester
equation
\begin{equation}
\label{XSylY} X\Lambda_0 -\Lambda_1 X=Y.
\end{equation}
Observe that  $\|Y\|\leq \|B\|(1+\|X\|^2)$.  Taking into account
\eqref{spLL}, Theorem \ref{TSylB} yields
$$
\frac{\|X\|}{1+\|X\|^2}\leq \frac{\pi}{2}\frac{\|B\|}{\dist(\omega,\Omega)}.
$$
Now the claim follows from the fact that $2\|X\|/(1+\|X\|^2)=\sin(2\theta)$
by \eqref{eq:NPD}.
\end{proof}
\begin{remark}
\label{R-tpi4}
Clearly, if $\pi\|V\|>\dist(\omega,\Omega)$, the estimate \eqref{s2t-apo} is
of no interest.
Suppose that
$\dist(\omega,\Omega)=\delta>0\text{\, and \,}
\pi\|V\|\leq\delta.$
In such a case \eqref{s2t-apo} does allow to obtain a bound for the
maximal angle $\theta$ but only provided the location of $\theta$
relative to $\frac{\pi}{4}$ is known. In particular, if it is known
that $\theta\leq\frac{\pi}{4}$ then \eqref{s2t-apo} implies the
upper bound
$\theta\leq\frac{1}{2}\arcsin\left(\frac{\pi\|V\|}{\delta}\right)$.
On the contrary, if it is known that $\theta\geq\frac{\pi}{4}$ then
\eqref{s2t-apo} yields the lower bound
$\theta\geq\frac{\pi}{2}-\frac{1}{2}\arcsin\left(\frac{\pi\|V\|}{\delta}\right)$.
\end{remark}

\begin{corollary}
\label{ThS2Tpr} Let $A$ be a (possibly unbounded) self-adjoint
operator on the Hilbert space $\fH$. Assume that $\fA$ is the
spectral subspace of $A$ associated with a Borel subset
$\sigma$ of its spectrum. Suppose that $V$ is a bounded self-adjoint
operator on $\fH$ and $L=A+V$ with $\Dom(L)=\Dom(A)$. Furthermore,
assume that $\fL$ is a reducing subspace of $L$.
If the subspaces $\fA$ and $\fL$ are in the acute-angle case then
\begin{equation}
\label{s2t-apr}
\dist(\sigma,\Sigma)\,\sin(2\theta)\leq \pi \|V\|,
\end{equation}
where $\Sigma=\spec(A)\setminus\sigma$ is the remainder of the
spectrum of $A$ and $\theta$ denotes the maximal angle between
$\fA$ and $\fL$.
\end{corollary}
\begin{proof}
Consider $A$ as the perturbation of the operator $L$, namely view
$A$ as $A=L+W$ with $W=-V$, and then the assertion follows from
Theorem \ref{ThS2Tpo}.
\end{proof}

\begin{remark}
\label{Rtet-pi4-a} Suppose that $\dist(\sigma,\Sigma)=d>0$ and\,
$\|V\|\leq c_{_{\pi/4}}d$,\, where $c_{_{\pi/4}}$ is given by
\eqref{cpi4}; observe that
$c_{_{\pi/4}}<\frac{1}{\pi}=0.318309\ldots$\, and, thus,
$\frac{\pi\|V\|}{d}<1$. Let $\omega=\spec(L)\cap\cO_{\|V\|}(\sigma)$
and $\fL=\Ran\bigl(\sE_L(\omega)\bigr)$. By Remark \ref{R-MSpi4}
under the condition  $\|V\|\leq c_{_{\pi/4}}d$ we have
$\theta(\fA,\fL)\leq\frac{\pi}{4}$ and, hence, \eqref{s2t-apr}
yields the bound (cf. Remark \ref{R-tpi4})
\begin{equation}
\label{s2t-pi4}
\theta(\fA,\fL)\leq\frac{1}{2}\arcsin\left(\frac{\pi\|V\|}{d}\right).
\end{equation}
\end{remark}
\smallskip

Notice that
\begin{equation}
\label{2consts}
\frac{4}{\pi^2+4}=0.288400\ldots\text{\,\, and \,\,}
c_{_{\pi/4}}>\frac{4}{\pi^2+4}.
\end{equation}
In Theorem~\ref{ThMain} below we will prove that for
$\frac{4}{4+\pi^2}d<\|V\|\leq\frac{1}{\pi}d$ there is a bound on
$\theta(\fA,\fL)$ tighter than estimate~\eqref{s2t-pi4}.

The estimates \eqref{s2t-apo} and \eqref{s2t-apr} we obtained in
Theorem \ref{ThS2Tpo} and Corollary \ref{ThS2Tpr} will be called the
\emph{generic} a posteriori and a priori $\sin2\theta$ bounds,
respectively. These estimates resemble the corresponding bounds from
the celebrated $\sin2\Theta$ theorems proven by C.\,Davis and
W.\,M.\,Kahan in \cite[Section 7]{DK70} for particular dispositions
of the sets $\omega$ and $\Omega$ or $\sigma$ and $\Sigma$.  Recall
that, when proving the $\sin2\Theta$ theorems, it is assumed in
\cite{DK70} that the convex hull of one of the sets $\omega$ and
$\Omega$ (resp., the convex hull of one of the sets $\sigma$ and
$\Sigma$) does not intersect the other set. An immediately visible
difference is that the constant $\pi$ shows up on the right-hand side parts
of the \emph{generic} $\sin2\theta$ estimates \eqref{s2t-apo} and
\eqref{s2t-apr} instead of the constant 2 in the Davis-Kahan
$\sin2\Theta$ theorems.

\section{A new rotation bound by multiply employing sin\,2$\theta$ estimate}
\label{S-StarBound}

Let $\{\vk_n\}_{n=0}^\infty$ be a number sequence with
\begin{equation}
\label{xns}
\vk_0=0,\quad \vk_n=\frac{4}{\pi^2+4}+\frac{\pi^2-4}{\pi^2+4}\vk_{n-1},\quad n=1,2,\ldots\,.
\end{equation}
One easily verifies that the general term of this sequence reads
\begin{equation}
\label{xn}
\vk_n=\frac{1-q^n}{2}, \quad n=0,1,2,\ldots,
\end{equation}
with
\begin{equation}
\label{alq}
q=\dfrac{\pi^2-4}{\pi^2+4}<1.
\end{equation}
Obviously, the sequence \eqref{xn} is strictly monotonously increasing and
it is bounded from above by $1/2$,
\begin{equation}
\label{xn12}
\vk_0<\vk_1<\vk_2<\ldots<\vk_n<\ldots<{1}/{2}.
\end{equation}
Moreover,
\begin{equation}
\label{limxn}
\lim_{n\to\infty} \vk_n={1}/{2}.
\end{equation}
Thus, this sequence produces a countable partition of the interval
$\bigl[0,\frac{1}{2}\bigr)$,
\begin{equation}
\label{Union}
\bigl[0,\mbox{$\frac{1}{2}$}\bigr)=\bigcup_{n=0}^\infty[\vk_n,\vk_{n+1}).
\end{equation}

Taking into account \eqref{Union}, with the
sequence $\{\vk_n\}_{n=0}^\infty$ we associate a function $M_\star(x)$,
$M_\star:\bigl[0,\frac{1}{2}\bigr)\to\bbR$, that is defined
separately on each elementary interval $[\vk_n,\vk_{n+1})$ by
\begin{equation}
\label{Mstar}
M_\star(x)\bigr|_{[\vk_n,\vk_{n+1})}=\frac{n}{2}\arcsin\left(\frac{4\pi}{\pi^2+4}\right)
+\frac{1}{2}\arcsin\left(\frac{\pi(x-\vk_n)}{1-2\vk_n}\right),
\,\,\, n=0,1,2,\ldots\,.\quad
\end{equation}
We note that the function $M_\star(x)$ may be equivalently written in the
form \eqref{MstarIn}.

\begin{proposition}
\label{PropM}
The function $M_\star(x)$ is continuous and continuously
differentiable on the interval $\bigl[0,\frac{1}{2}\bigr)$.
Furthermore, this function is strictly monotonously increasing on
$\bigl[0,\frac{1}{2}\bigr)$ and
\mbox{$\lim\limits_{x\,\,\uparrow\,\frac{1}{2}}M_\star(x)=+\infty$.}
\end{proposition}
\begin{proof}
Clearly, one needs to check continuity and continuous
differentiability of $M_\star(x)$ only at the points $\vk_n$,
$n=1,2,\ldots$\,. Given $n\in\bbN$, by
\eqref{Mstar} for $x\in[\vk_{n-1},\vk_n)$ we have
\begin{equation}
\label{Mxn}
M_\star(x)=\frac{n-1}{2}\arcsin\left(\frac{4\pi}{\pi^2+4}\right)
+\frac{1}{2}\arcsin\left(\frac{\pi(x-\vk_{n-1})}{1-2\vk_{n-1}}\right),\quad
x\in[\vk_{n-1},\vk_n).
\end{equation}
Note that \eqref{xns} yields
\begin{align}
\nonumber
\vk_n-\vk_{n-1}&=\frac{4}{\pi^2+4}+\frac{\pi^2-4}{\pi^2+4}\vk_{n-1}-\vk_{n-1}\\
\label{help}
           &=\frac{4}{\pi^2+4}\left(1-2\vk_{n-1}\right),
           \quad \text{for all \,}n\in\bbN.
\end{align}
By \eqref{help} one observes that
\begin{equation}
\label{Mlc}
\frac{x-\vk_{n-1}}{1-2\vk_{n-1}}\,\,\mathop{\longrightarrow}\limits_{x\to
\vk_n}\,\,\frac{\vk_n-\vk_{n-1}}{1-2\vk_{n-1}}=\frac{4}{\pi^2+4}
\end{equation}
and then from \eqref{Mxn} it follows that
\begin{equation}
\label{Mlimleft}
\lim\limits_{x\,\uparrow \vk_n}M_\star(x)=
\frac{n}{2}\arcsin\left(\frac{4\pi}{\pi^2+4}\right).
\end{equation}
Meanwhile, by its definition \eqref{Mstar} on the interval
$x\in[\vk_n,\vk_{n+1})$, the function $M_\star(x)$ is right-continuous at
$x=\vk_n$ and
$$
\frac{n}{2}\arcsin\left(\frac{4\pi}{\pi^2+4}\right)=
\lim\limits_{x\,\downarrow \vk_n}M_\star(x)=M_\star(\vk_n).
$$
Hence, \eqref{Mlimleft} yields continuity of
$M_\star(x)$ at $x=\vk_n$.

Using equality \eqref{help} one more time, one verifies by inspection
that for any $n\in\bbN$ the left and right limiting values
$$
\lim\limits_{x\,\uparrow \vk_n}M'_\star(x)=\frac{1}{2}\,
\frac{1}{\sqrt{1-\dfrac{\pi^2(\vk_n-\vk_{n-1})^2}{(1-2\vk_{n-1})^2}}}\,
\frac{\pi}{1-2\vk_{n-1}}
\quad\text{and}\quad
\lim\limits_{x\,\downarrow \vk_n}M'_\star(x)=\frac{1}{2}\frac{\pi}{1-2\vk_n}
$$
of the derivative $M'_\star(x)$ as $x\to \vk_n$ are equal to each other.
Hence,  for any $n\in\bbN$ the derivative $M'_\star(x)$ is
continuous at $x=\vk_n$ and then it is continuous on
$\bigl[0,\frac{1}{2}\bigr)$.

Obviously, $M'_\star(x)>0$ for all $x\in\bigl(0,\frac{1}{2}\bigr)$,
which means that the function $M_\star(x)$ is strictly monotonously
increasing on $\bigl[0,\frac{1}{2}\bigr)$. Then, given any
$n\in\bbN$,  from \eqref{Mstar} it follows that
\mbox{$M_\star(x)\geq n C_0$} with
$C_0=\frac{1}{2}\arcsin\left(\frac{4\pi}{\pi^2+4}\right)>0 $
whenever $x\geq \vk_n$. Taking into account \eqref{limxn}, this
implies that \mbox{$M_\star(x)\to\infty$} as \mbox{$x\to\frac{1}{2}$},
completing the proof.
\end{proof}

\begin{remark}
\label{R-cstar}
Since the function $M_\star(x)$ is continuous and strictly monotonous on
$\bigl[0,\frac{1}{2}\bigr)$ and $M_\star(0)=0$,
$\lim\limits_{x\,\,\uparrow\,\frac{1}{2}}M_\star(x)=+\infty$,
there is a unique number $c_\star\in\bigl[0,\frac{1}{2}\bigr)$ such that
\begin{equation}
\label{Msolv}
M_\star(c_\star)=\frac{\pi}{2}.
\end{equation}
An explicit numerical evaluation of
\begin{equation}
\label{C0}
C_0=\frac{1}{2}\arcsin\left(\frac{4\pi}{\pi^2+4}\right)
\end{equation}
shows that
$C_0=\frac{\pi}{2}\times 0.360907\ldots$\,. Thus, $2C_0<\frac{\pi}{2}$
while $3C_0>\frac{\pi}{2}$. Hence, $c_\star\in[\vk_2,\vk_3)$ and by \eqref{Mstar},
with $n=2$, equation \eqref{Msolv} turns into
\begin{equation}
\label{eqcs}
\arcsin\left(\frac{4\pi}{\pi^2+4}\right)
+\frac{1}{2}\arcsin\left(\frac{\pi(c_\star-\vk_2)}{1-2\vk_2}\right)=\frac{\pi}{2},
\end{equation}
with $\vk_2$ being equal (see \eqref{xn})  to
\begin{equation}
\label{x2}
\vk_2=
\frac{8\pi^2}{(\pi^2+4)^2}=0.410451\ldots \quad\left(>\frac{1}{\pi}\right).
\end{equation}
One verifies by inspection that the solution $c_\star$ to equation
\eqref{eqcs}, and hence to equation \eqref{Msolv}, reads
\begin{equation}
\label{cstar}
c_\star=16\,\,\frac{\pi^6-2\pi^4+32\pi^2-32}{(\pi^2+4)^4}=0.454169\ldots\,\,.
\end{equation}
\end{remark}

\begin{proposition}
{\rm (Optimality of the function $M_\star$.)}
Let $\{\mu_n\}_{n=0}^\infty\subset\bigl[0,\frac{1}{2}\bigr)$ be
an arbitrary mono\-to\-nous\-ly increasing number sequence such that
\begin{equation}
\label{yns}
\mu_0=0 \quad\text{and}\quad
0<\frac{\pi(\mu_{n}-\mu_{n-1})}{1-2\mu_{n-1}}\leq 1\quad\text{for \,}
n\geq 1.
\end{equation}
Assume that $\sup\nolimits_n \mu_n:=\mu_{\rm sup}$ and introduce the
function $F:\,[0,\mu_{\rm sup})\to \bbR$ by
\begin{align}
\label{F1}
F(x)\bigr|_{[0,\mu_1)}&=\frac{1}{2}\arcsin(\pi x),\\
\label{Fn}
F(x)\bigr|_{[\mu_n,\mu_{n+1})}&=
\frac{1}{2}\sum\limits_{j=1}^{n}\arcsin\left(\frac{\pi(\mu_j-\mu_{j-1})}{1-2\mu_{j-1}}\right)
+\frac{1}{2}\arcsin\left(\frac{\pi(x-\mu_n)}{1-2\mu_n}\right),\quad n\geq 1,
\end{align}
The function $M_\star(x)$ is optimal in the sense that if
$\{\mu_n\}_{n=0}^\infty$ does not coincide with the sequence
\eqref{xns}, then there always exists an open interval
$\cF\subset(0,\mu_{\rm sup})$ such that $F(x)>M_\star(x)$ for all
$x\in\cF$.
\end{proposition}

\begin{proof}
First, we remark that if one chooses $\mu_n=\vk_n$, $n=0,1,2\ldots\,,$ where
$\{\vk_n\}_{n=0}^\infty$ is the sequence \eqref{xns}, then  $F$
coincides with $M_\star$. If
$\{\mu_n\}_{n=0}^\infty\neq\{\vk_n\}_{n=0}^\infty$ then there is
$m\in\bbN$ such that $\mu_n=\vk_n$ for all $n<m$ and $\mu_m\neq \vk_m$.
Since, $\mu_{m-1}=\vk_{m-1}$, there are two options: either
\begin{equation}
\label{case-1-1}
\vk_{m-1}<\mu_m<\min\{\vk_m,\mu_{m+1}\}
\end{equation}
or
\begin{equation}
\label{case-1-2}
\vk_m<\mu_m \quad\text{(and \,}\mu_m<\mu_{m+1}).
\end{equation}
By \eqref{Mstar} and \eqref{F1}, \eqref{Fn} in  the case \eqref{case-1-1}
equality $M_\star(x)=F(x)$ holds for all $x\in[0,\mu_m]$. For
$x\in\bigl(\mu_m,\min\{\vk_m,\mu_{m+1}\}\bigr)$ we have
\begin{align*}
M_\star(x)&=(m-1)C_0
+\frac{1}{2}\arcsin\left(\frac{\pi(x-\vk_{m-1})}{1-2\vk_{m-1}}\right), \\
F(x)&=
(m-1)C_0
+\frac{1}{2}\arcsin\left(\frac{\pi(\mu_m-\vk_{m-1})}{1-2\vk_{m-1}}\right)
+\frac{1}{2}\arcsin\left(\frac{\pi(x-\mu_m)}{1-2\mu_m}\right),
\end{align*}
where $C_0$ is given by \eqref{C0}. Having explicitly computed the
derivatives of $M_\star(x)$ and $F(x)$ for
$x\in\bigl(\mu_1,\min\{\vk_m,\mu_{m+1}\bigr)$, one obtains that
$F'(x)>M'_\star(x)$ whenever $ \mu_m<x<\xi_m$, where
$$
\xi_m=\min\left\{\mu_{m+1},\mbox{$\frac{1}{2\pi^2}
[4+(\pi^2-4)(\vk_{m-1}+\mu_m)]$}\right\}.
$$
Notice that $\xi_m>\mu_m$ since $\mu_{m+1}>\mu_m$ and
\begin{align*}
\frac{1}{2\pi^2}
[4+(\pi^2-4)(\vk_{m-1}+\mu_m)]-\mu_m&
=\frac{\pi^2+4}{2\pi^2}\left(\frac{4}{\pi^2+4}+\frac{\pi^2-4}{\pi^2+4}\vk_{m-1}-\mu_m\right)\\
&=\frac{\pi^2+4}{2\pi^2}(\vk_m-\mu_m)\\
&>0
\end{align*}
(see \eqref{xns} and \eqref{case-1-1}). Observing that $M_\star(\mu_m)=F(\mu_m)$,
one concludes that $F(x)>M_\star(x)$ at least for all $x$ from the open
interval $\cF=(\mu_m,\xi_m)$.

If \eqref{case-1-2} holds then $M_\star(x)=F(x)$ for $x\in[0,\vk_m]$.
At the same time, for
\begin{equation}
\label{vkxm}
\vk_m< x<\min\{\vk_{m+1},\mu_m\}
\end{equation}
we have
\begin{align}
\label{eqFM1}
M_\star(x)&=
(m-1)C_0
+\frac{1}{2}\arcsin\left(\frac{\pi(\vk_m-\vk_{m-1})}{1-2\vk_{m-1}}\right)
+\frac{1}{2}\arcsin\left(\frac{\pi(x-\vk_m)}{1-2\vk_m}\right), \\
\label{eqFM2}
F(x)&=(m-1)C_0
+\frac{1}{2}\arcsin\left(\frac{\pi(x-\vk_{m-1})}{1-2\vk_{m-1}}\right).
\end{align}
One verifies by inspection that under condition \eqref{vkxm} the requirement
$F'(x)>M'_\star(x)$ is equivalent to
\begin{align*}
x>&\frac{1}{2\pi^2}[4+(\pi^2-4)(\vk_{m-1}+\vk_m)]\\
&=\frac{\pi^2+4}{2\pi^2}\left(\frac{4}{\pi^2+4}+
\frac{\pi^2-4}{\pi^2+4}\vk_{m-1}+\frac{\pi^2-4}{\pi^2+4}\vk_{m}\right)\\
&=\frac{\pi^2+4}{2\pi^2}\left(\vk_m+\frac{\pi^2-4}{\pi^2+4}\vk_{m}\right)\\
&=\vk_m,
\end{align*}
by taking into account relations \eqref{xns} at the second step.
That is, \eqref{vkxm} implies $F'(x)>M'_\star(x)$. From
$M_\star(\vk_m)=F(\vk_m)$ it then follows that $F(x)>M_\star(x)$ at
least for all $x$ from the open interval
$\cF=(\vk_m,\min\{\vk_{m+1},\mu_m\})$. The proof is complete.
\end{proof}

Finally, we turn to the main result of this work.

\begin{theorem}
\label{ThMain} Given a (possibly unbounded) self-adjoint operator
$A$ on the Hilbert space $\fH$, assume that a Borel set
$\sigma\subset\bbR$ is an isolated component of the spectrum of $A$,
i.e. $\sigma\subset\spec(A)$ and
\begin{equation*}
\dist(\sigma,\Sigma)=d>0,
\end{equation*}
where $\Sigma=\spec(A)\setminus\sigma$. Let $V$ be a
bounded self-adjoint operator on $\fH$ such that
\begin{equation}
\label{Vd12}
\|V\|<{d}/{2}
\end{equation}
and let $L=A+V$ with $\Dom(L)=\Dom(A)$. Then the maximal angle
$\theta(\fA,\fL)$ between the spectral subspaces $\fA=\Ran
\bigl(\sE_A(\sigma)\bigr)$ and $\fL=\Ran\bigl(\sE_L(\omega)\bigr)$
of $A$ and $L$ associated with their respective spectral subsets
$\sigma$ and $\omega=\spec(L)\cap\cO_{\|V\|}(\sigma)$ satisfies the
bound
\begin{equation}
\label{TetaM}
\theta(\fA,\fL)\leq M_\star\left(\frac{\|V\|}{d}\right),
\end{equation}
where the estimating function $M_\star(x)$,
$x\in\bigl[0,\frac{1}{2}\bigr)$, is defined by \eqref{Mstar}.
In particular, if
\begin{equation}
\label{Vcsd}
\|V\|<c_\star\, d,
\end{equation}
where $c_\star$ is given by \eqref{cstar}, then the subspaces $\fA$
and $\fL$ are in the acute-angle case, i.e.
$\theta(\fA,\fL)<\frac{\pi}{2}$.
\end{theorem}

\begin{proof}
Throughout the proof we assume that $\|V\|\neq 0$ and set
\begin{equation}
\label{xVd}
x=\frac{\|V\|}{d}.
\end{equation}
The assumption \eqref{Vd12} implies
$x<\frac{1}{2}$. Hence, there is a number $n\in\bbN\cup\{0\}$ such
that $x\in[\vk_n,\vk_{n+1})$ with $\vk_n$ and $\vk_{n+1}$ the consecutive
elements of the sequence \eqref{xns}.

For $n=0$ the bound \eqref{TetaM} holds by Remark \ref{Rtet-pi4-a}
since in this case $\frac{\|V\|}{d}<\frac{4}{\pi^2+4}<c_{_{\pi/4}}$
and $M_\star(x)=\frac{1}{2}\arcsin(\pi x)$\, (see \eqref{2consts},
\eqref{xns}, and \eqref{Mstar}).

In the case where $n\geq 1$ we introduce the operators
$$
V_j=\vk_j \frac{d}{\|V\|}V\quad\text{and}\quad L_j=A+V_j,\quad\Dom(L_j)=\Dom(A), \quad
j=0,1,\ldots,n,
$$
where $\vk_j$ are elements of the sequence \eqref{xns}. Since $\vk_j<\vk_n$ for
$j<n$ and $\vk_n\leq x$, from \eqref{xVd} it follows that
$$
\|V_j\|=\frac{\vk_j}{x}\|V\|\leq \|V\|<\frac{d}{2}, \quad j=0,1,\dots,n.
$$
Therefore, the spectrum of the (self-adjoint) operator $L_j$
consists of the two disjoint components
$$
\omega_j=\spec(L_j)\cap\cO_{\|V_j\|}(\sigma)\quad\text{and}\quad
\Omega_j=\spec(L_j)\cap\cO_{\|V_j\|}(\Sigma), \quad j=0,1,\dots,n.
$$
Moreover,
\begin{equation}
\label{delj}
\delta_j:=\dist(\omega_j,\Omega_j)\geq d-2\|V_j\|=d(1-2\vk_j), \quad j=0,1,\dots,n.
\end{equation}
By $\fL_j$ we will denote the spectral subspace of $L_j$ associated
with its spectral component $\omega_j$, i.e.
$\fL_j=\Ran\bigl(\sE_{L_j}(\omega_j)\bigr)$. Notice that $L_0=A$,
$\omega_0=\sigma$, and $\fL_0=\fA$.

For $0\leq j\leq n-1$ the operator $L_{j+1}$ may be viewed as a perturbation
of the operator $L_j$, namely
$$
L_{j+1}=L_j+W_j, \quad j=0,1,\dots,n-1,
$$
where
$$
W_j:=V_{j+1}-V_j=(\vk_{j+1}-\vk_j)\frac{d}{\|V\|}V.
$$
Similarly, we write
$$
L=L_n+W, \quad\text{with \,\,\,} W:=V-V_n=(x-\vk_n)\frac{d}{\|V\|}V.
$$
One easily verifies that
\begin{align}
\label{omegaj1}
\omega_{j+1}&=\spec(L_{j+1})\cap\cO_{\|W_j\|}(\omega_j),\,\,
j=0,1,\ldots,\mbox{$n-1$},\\
\label{sigmap}
\omega&=\spec(L)\cap\cO_{\|W\|}(\omega_n).
\end{align}
By taking into account first \eqref{delj} and
then \eqref{help}, one observes that
\begin{align}
\label{Wjn}
\frac{\|W_j\|}{\delta_j}&=\frac{(\vk_{j+1}-\vk_j)d}{\delta_j}
\leq\frac{\vk_{j+1}-\vk_j}{1-2\vk_j}=\frac{4}{\pi^2+4},\quad j=0,1,\dots,n-1,
\end{align}
and
\begin{align}
\label{Wn} \frac{\|W\|}{\delta_n}&=\frac{(x-\vk_n)d}{\delta_n}
\leq\frac{x-\vk_n}{1-2\vk_n}\qquad
\left(<\frac{\vk_{n+1}-\vk_n}{1-2\vk_n}=\frac{4}{\pi^2+4}\right).
\end{align}
Recall that $\frac{4}{\pi^2+4}<c_{_{\pi/4}}$ (see \eqref{cpi4} and
\eqref{2consts}). Thus, by \eqref{omegaj1}--\eqref{Wn} Remark
\ref{Rtet-pi4-a} applies to any of the pairs $(\fL_j,\fL_{j+1})$,
$j=0,1,\ldots,n-1$, and $(\fL_n,\fL)$, which means that
\begin{align}
\label{tetjj1}
\theta(\fL_j,\fL_{j+1})&\leq
\frac{1}{2}\arcsin\left(\frac{\pi\|W_j\|}{\delta_j}\right)
\leq\frac{1}{2}\arcsin\left(\frac{4\pi}{\pi^2+4}\right),\quad j=0,1,\dots,n-1,\\
\label{tetnL}
\theta(\fL_n,\fL)&\leq
\frac{1}{2}\arcsin\left(\frac{\pi\|W\|}{\delta_n}\right)
\leq\frac{1}{2}\arcsin\left(\frac{\pi(x-\vk_n)}{1-2\vk_n}\right).
\end{align}
Meanwhile, by using the triangle inequality for maximal angles
between subspaces (see Lem\-ma~\ref{L-triangle}) one obtains, step by step,
\begin{alignat}{2}
\nonumber
\theta(\fL_0,\fL)&\leq \theta(\fL_0,\fL_1)+\theta(\fL_1,\fL)
&(\text{if }n\geq 1)\,\,\,\\
\nonumber
&\leq \theta(\fL_0,\fL_1)+\theta(\fL_1,\fL_2)+\theta(\fL_2,\fL)
\qquad\qquad   &(\text{if }n\geq 2)\,\,\,\\
\nonumber
&\leq \quad \cdots  & \cdots\qquad \\
\label{trianglen}
&\leq \sum\limits_{j=0}^{n-1}
\theta(\fL_j,\fL_{j+1})+\theta(\fL_n,\fL) & (\text{if }n\geq 1).
\end{alignat}
Combining \eqref{trianglen} with \eqref{tetjj1} and \eqref{tetnL}
results just in the bound \eqref{TetaM}, taking into account
equality $\fL_0=\fA$ and the definitions \eqref{xVd} of $x$ and
\eqref{Mstar} of $M_\star(x)$. Finally, by Proposition \ref{PropM}
and Remark \ref{R-cstar} one concludes that under condition
\eqref{Vcsd} the subspaces $\fA$ and $\fL$ are in the acute-angle
case. This completes the proof.
\end{proof}

\begin{figure}
\begin{center}
{\includegraphics[angle=0,width=10.cm]{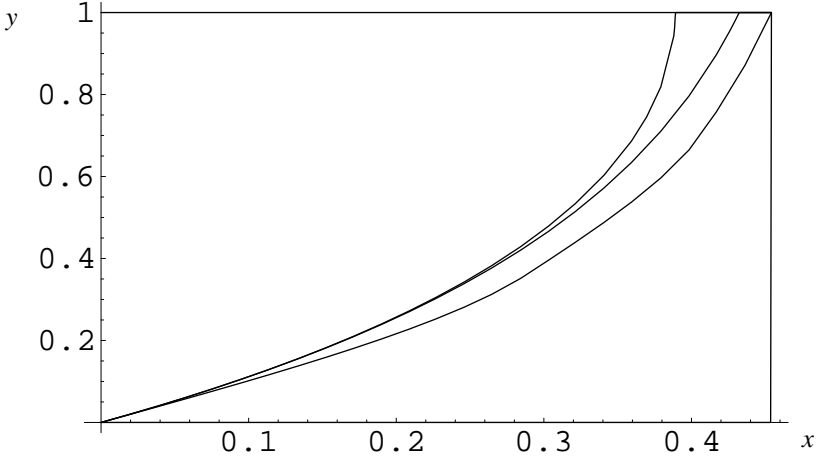}}
\end{center}

\vspace*{-6.3cm}
\hspace*{-10.5cm} $y$
\vspace*{5.2cm}

\hspace*{10.5cm}$x$

\bigskip

\caption{Graphs of the functions $\frac{2}{\pi}M_{\rm KMM}(x)$,
$\frac{2}{\pi}M_{\rm MS}(x)$, and $\frac{2}{\pi}M_\star(x)$ while
their values do not exceed 1. The upper curve depicts the graph
of $\frac{2}{\pi}M_{\rm KMM}(x)$ for $x\leq c_{_{\rm KMM}}$, the
intermediate curve represents the graph of $\frac{2}{\pi}M_{\rm
MS}(x)$ for $x\leq c_{_{\rm MS}}$, and the lower curve is the graph of
$\frac{2}{\pi}M_\star(x)$ for $x\leq c_\star$.} \label{figure1}
\end{figure}

\begin{remark}
\label{R-final} Recall that the previously known estimating
functions for $\theta(\fA,\fL)$ are those of references \cite{KMM1}
and \cite{MakS10}, namely the functions $M_{_{\rm
KMM}}\left(\frac{\|V\|}{d}\right)$ (see Remark \ref{R-KMMbound}) and
$M_{_{\rm MS}}\left(\frac{\|V\|}{d}\right)$ (see Remark
\ref{R-MSbound}). One verifies by inspection that the derivatives
$M'_{_{\rm KMM}}(x)$, $M'_{_{\rm MS}}(x)$, and $M'_\star(x)$ of the estimating
functions $M_{_{\rm KMM}}(x)$, $M_{_{\rm MS}}(x)$, and $M_\star(x)$
possess the following properties
\begin{align}
\label{MLM1}
M'_{_{\rm MS}}(x)&<M'_{_{\rm KMM}}(x)\quad\text{for all \,}x\in(0,c_{_{\rm KMM}}),\\
\label{MLM2}
M'_\star(x)&<M'_{_{\rm MS}}(x) \text{\,\, for all \,}
x\in\mbox{$\bigl[0,\frac{1}{2}\bigr)$}\setminus\{\vk_n\}_{n=0}^\infty,\\
\label{MLM3}
M'_\star(\vk_n)&=M'_{_{\rm MS}}(\vk_n),\quad n=0,1,2,\ldots\,,
\end{align}
where $c_{_{\rm KMM}}$ is given by \eqref{cKMM} and $\{\vk_n\}_{n=0}^\infty$ is
the sequence \eqref{xns}. Since
$$
M_{_{\rm KMM}}(0)=M_{_{\rm MS}}(0)=M_\star(0)=0,
$$
from \eqref{MLM1}--\eqref{MLM3} it follows that
\begin{equation*}
M_\star(x)<M_{_{\rm MS}}(x)<M_{_{\rm KMM}}(x)\quad
\text{for all \,}x\in(0,c_{_{\rm KMM}}]
\end{equation*}
and
\begin{equation*}
M_\star(x)<M_{_{\rm MS}}(x)\quad
\text{for all \,}x\in\mbox{$\bigl[c_{_{\rm KMM}},\frac{1}{2}\bigr)$}.
\end{equation*}
Thus, the bound \eqref{TetaM} is stronger than both the previously
known bounds \cite{KMM1,MakS10} for $\theta(\fA,\fL)$, in particular
it is stronger than the best of them, the bound \eqref{MSbound},
established in \cite{MakS10}.

For convenience of the reader, the graphs of the estimating
functions $M_\star(x)$, $M_{_{\rm MS}}(x)$, and $M_{_{\rm KMM}}(x)$,
all the three divided by \mbox{\small$\pi/2$}, are plotted
in Fig. \ref{figure1}. Plotting of the function
$\frac{2}{\pi}\,M_{_{\rm KMM}}(x)$ is naturally restricted to its
domain $[0,c_{_{\rm KMM}}]$. The functions $M_\star(x)$ and
$M_{_{\rm MS}}(x)$ are plotted respectively for $x\in[0,c_\star]$
and $x\in[0,c_{_{\rm MS}}]$ where $c_\star$ is given by
\eqref{cstar} and $c_{_{\rm MS}}$ by \eqref{cMS}.
\end{remark}

We conclude this section with an a posteriori result that is an
immediate corollary to The\-o\-rem~\ref{ThMain}.

\begin{theorem}
\label{ThMApo} Assume that $A$ and $V$ are self-adjoint operators on the
Hilbert space $\fH$. Let $V\in\cB(\fH)$ and $L=A+V$ with
$\Dom(L)=\Dom(A)$. Assume, in addition, that $\omega$ is an isolated
component of the spectrum of $L$, i.e.
$\dist(\omega,\Omega)=\delta>0$, where
$\Omega=\spec(L)\setminus\omega$, and suppose that $\|V\|<\delta/2$. Then the
maximal angle $\theta(\fA,\fL)$ between the spectral subspaces
$\fA=\Ran \bigl(\sE_A(\sigma)\bigr)$ and
$\fL=\Ran\bigl(\sE_L(\omega)\bigr)$ of $A$ and $L$ associated with
their respective spectral components
\mbox{$\sigma=\spec(A)\cap\cO_{\|V\|}(\omega)$} and $\omega$
satisfies the bound
\begin{equation}
\label{TetaMApo}
\theta(\fA,\fL)\leq M_\star\left(\frac{\|V\|}{\delta}\right)
\end{equation}
with $M_\star$ given by \eqref{Mstar}. In particular, if
$\|V\|<c_\star\, \delta$ where $c_\star$ is given by \eqref{cstar},
the subspaces $\fA$ and $\fL$ are in the acute-angle case.
\end{theorem}

\begin{proof}
Do exactly the same step as we did in the proof of Corollary \ref{ThS2Tpr}:
Represent $A$ as $A=L+W$ with $W=-V$. Then the assertion follows from
Theorem \ref{ThMain}.
\end{proof}

\begin{remark}
As usually, let $\Sigma=\spec(A)\setminus\sigma$ and
$d=\dist(\sigma,\Sigma)$. Suppose that both the `a priori' and `a
posteriori' distances $d$ and $\delta$ are known. Then, depending on
which of the distances $d$ and $\delta$ is larger, one may choose
among the bounds \eqref{TetaM} and \eqref{TetaMApo} the stronger one:
\begin{equation*}
\theta(\fA,\fL)\leq M_\star\left(\frac{\|V\|}{\max\{d,\delta\}}\right).
\end{equation*}
\end{remark}

\section{Quantum harmonic oscillator under a bounded perturbation}
\label{SecExHO}

In this section we apply the results of the previous section to the
$N$-dimensional isotropic quantum harmonic oscillator under a
bounded self-adjoint perturbation.

Let $\fH=L_2(\bbR^N)$ for some $N\in\bbN$. Under the assumption
that the units are chosen such that the reduced Planck constant, mass
of the particle, and the angular frequency are all equal to one,
the Hamiltonian of the isotropic quantum harmonic oscillator is
given by
\begin{equation}
\label{Aho}
\begin{array}{l}
(Af)(x)=-\frac{1}{2}\Delta f(x)+\frac{1}{2}|x|^2f(x),\\[0.5em]
\dom(A)=\biggl\{f\in W^2_2(\bbR^N)\,\,\biggl|\,\, \displaystyle\int_{\bbR^N}
dx\; |x|^4|f(x)|^2<\infty\biggr\},
\end{array}
\end{equation}
where $\Delta$ is the Laplacian and $W_2^2(\bbR^N)$ denotes
the Sobolev space of $L_2(\bbR^N)$-functions that have their
second partial derivatives in $L_2(\bbR^N)$.

It is well known that the Hamiltonian $A$ is a self-adjoint operator
in $L_2(\bbR^N)$. Its spectrum consists of eigenvalues of
the form
\begin{equation}
\label{lamN}
 \lambda_n = n+ N/2, \quad n=0,1,2, \dots,
\end{equation}
whose multiplicities $m_n$ are given by the binomial coefficients
(see, e.g., \cite{LJVdJ2008} and references therein)
\begin{equation}
\label{mun}
m_n=\left(\begin{array}{c}N+n-1\\n\end{array}\right), \quad n=0,1,2, \dots.
\end{equation}
For $n$ even, the corresponding eigenfunctions $f(x)$ are
symmetric with respect to space reflection $x\,\mapsto-x$
(i.e.\ $f(-x)=f(x)$). For $n$ odd, the
eigenfunctions are anti-symmetric (i.e.\
$f(-x)=-f(x)$). Thus, if one partitions the spectrum
$\spec(A) = \sigma\cup\Sigma$ into the two parts
$$
\sigma=
\{n+N/2\,\,\bigl|\,\, n=0,2,4,\dots\} \quad\text{and}\quad
\Sigma=
\{n+N/2\,\,\bigl|\,\,n=1,3,5\ldots\},
$$
then the complementary subspaces
\begin{equation}
\label{Hho} \fA=L_{2,\textrm{even}}(\bbR^N), \quad
\fA^\perp=L_{2,\textrm{odd}}(\bbR^N)
\end{equation}
of symmetric and anti-symmetric functions are the
spectral subspaces of $A$ corresponding to the spectral components
$\sigma$ and $\Sigma$, respectively. Clearly,
$$
d=\dist(\sigma,\Sigma)=1.
$$

Let $V$ be an arbitrary bounded self-adjoint operator
on $L_2(\bbR^N)$ such that
\begin{equation}
\label{dV2H}
\|V\|<1/2.
\end{equation}
The perturbed oscillator
Hamiltonian $L=A+V$, $\Dom(L)=\Dom(A)$, is self-adjoint and its
spectrum remains discrete. Moreover, the closed $\|V\|$-neighborhood
$\cO_{\|V\|}(\lambda_n )$ of the eigenvalue \eqref{lamN} of $A$
contains exactly $m_n$ eigenvalues $\lambda'_{n,k}$,
$k=1,2,\ldots,m_n$, of $L$, counted with multiplicities, where $m_n$
is given by \eqref{mun} (see, e.g., \cite[Section V.4.3]{Kato}).

Further, assume that the following stronger condition holds:
\begin{equation}
\label{VcsH}
\|V\|<c_\star,
\end{equation}
where $c_\star$ is given by \eqref{cstar}. Let $\fL$ be the spectral
subspace of the perturbed Hamiltonian $L$ associated with the
spectral subset $\omega=\spec(L)\cap\cO_{\|V\|}(\sigma)$. Theorem
\ref{ThMain} ensures that under condition \eqref{VcsH} the
unperturbed and perturbed spectral subspaces $\fA$ and $\fL$ are in
the acute-angle case. Moreover, the maximal angle $\theta(\fA,\fL)$
between these subspaces satisfies the bound
\begin{equation}
\label{tALM}
\theta(\fA,\fL)\leq M_\star(\|V\|),
\end{equation}
where $M_\star$ stands for the function given by \eqref{Mstar}.

Obviously, under condition \eqref{VcsH} the orthogonal complement
$\fL^\perp$ is the spectral subspace of $L$ associated with the
spectral set $\Omega=\spec(L)\cap\cO_{\|V\|}(\Sigma)$. By Remark
\ref{R-tPQperp}, $\theta(\fA^\perp,\fL^\perp)=\theta(\fA,\fL)$.
Hence the maximal angle between the subspaces $\fA^\perp$ and
$\fL^\perp$ satisfies the same bound \eqref{tALM}, i.e.
$\theta(\fA^\perp,\fL^\perp)\leq M_\star(\|V\|)$. Surely, this can
also be seen directly from Theorem \ref{ThMain}.

{\small
\vspace*{2mm} \noindent {\bf Acknowledgments.} This work was
supported by the Deutsche For\-sch\-ungs\-gemeinschaft, by the
Heisenberg-Landau Program, and by the Russian Foundation for Basic
Research. A.\,K.\,Mo\-to\-vi\-lov gratefully acknowledges the kind
hospitality of the Institut f\"ur Angewandte Mathematik,
Universit\"at Bonn, where a part of this research has been
conducted.
}


\end{document}